\documentclass[11pt]{amsart}

\usepackage{amsthm}
\usepackage{amsmath}
\usepackage{amssymb}
\usepackage{enumerate}
\usepackage{graphicx}
\usepackage[hidelinks,pdftex]{hyperref}
\usepackage{booktabs}
\usepackage{color}
\usepackage[dvipsnames]{xcolor}
\usepackage{import}
\usepackage{tikz-cd}
\usepackage{microtype}

\AtBeginDocument{%
   \def\MR#1{}
}

\numberwithin{equation}{section}
\numberwithin{figure}{section}

\theoremstyle{plain}
\newtheorem{theorem}[equation]{Theorem}
\newtheorem{corollary}[equation]{Corollary}
\newtheorem{lemma}[equation]{Lemma}

\newtheorem{proposition}[equation]{Proposition}

\newtheorem*{namedtheorem}{\theoremname}
\newcommand{\theoremname}{testing}

\theoremstyle{definition}
\newtheorem{definition}[equation]{Definition}

\newtheorem{example}[equation]{Example}

\newcommand{\HH}{{\mathbb{H}}}

\newcommand{\calD}{{\mathcal{D}}}

\newcommand{\refthm}[1]{Theorem~\ref{Thm:#1}}
\newcommand{\reflem}[1]{Lemma~\ref{Lem:#1}}
\newcommand{\refprop}[1]{Proposition~\ref{Prop:#1}}
\newcommand{\refcor}[1]{Corollary~\ref{Cor:#1}}

\newcommand{\refdef}[1]{Definition~\ref{Def:#1}}

\newcommand{\reffig}[1]{Figure~\ref{Fig:#1}}

\newcommand{\bdy}{\partial}
\newcommand{\voct}{v_{\mathrm{oct}}}
\newcommand{\vtet}{v_{\mathrm{tet}}}

\newcommand{\vol}{\operatorname{vol}}

\newcommand{\cut}{\backslash\backslash}

\title[Fully augmented links in doubled 3-manifolds]{Geometry of fully augmented links in doubled 3-manifolds}

\author{Jessica S.~Purcell}
\author{Corbin Reid}
\author{John Stewart}

\begin{document}

\begin{abstract}
Classical fully augmented links have explicit hyperbolic geometry, and have diagrams on the 2-sphere in the 3-sphere. We generalise to construct fully augmented links projected to the reflection surface of any 3-manifold obtained by doubling a compact 3-manifold and show that the results of the classical setting extend to these links. When the resulting manifolds are hyperbolic, we find bounds on their cusp shapes and volumes. Note these links include virtual fully augmented links, and thus our bounds apply to such links when they are hyperbolic. 
\end{abstract}

\maketitle

\section{Introduction}\label{Sec:Intro}

Fully augmented links in the 3-sphere are a class of links that exhibit tractable geometric structures, arising from the geometry of a circle packing on the 2-sphere.  Any knot or link in the 3-sphere can be obtained by Dehn filling a fully augmented link, which has led to important applications; see for example~\cite{Lackenby:Vol, BlairFuterTomova, FuterPurcell:Exceptional, FKP:DehnFilling, KalfagianniLee, Millichap, Purcell:CuspShapes, Purcell:AugLinksSurvey}.

However, all these applications are for knots and links with a diagram projected onto a 2-sphere, embedded in the 3-sphere. There are many important questions in hyperbolic knot theory concerning knots and links that do not have classical diagrams of this form. For example, \emph{virtual links} are embedded in a 3-manifold homeomorphic to a thickened surface $\Sigma\times[-1,1]$, with a diagram on $\Sigma\times\{0\}$. The hyperbolic geometry of these links is of increasing interest; see for example~\cite{Kuperberg, adamsetal:2021Generalized, CKPBiperiodic}. Additionally, \emph{shadow links} lie on a Heegaard surface in a doubled handlebody, and the hyperbolic geometry of such links has connections to quantum topology; see for example~\cite{CostantinoThurston, BDKY, Kumar}. Over many years, there has been interest in the hyperbolic geometry of alternating links embedded on surfaces besides the 2-sphere in manifolds besides the 3-sphere; for example~\cite{Adams:Toroidally, Hayashi, HowiePurcell}. The purpose of this paper is to show that many of the hyperbolic geometric results for classical fully augmented links can be extended directly to classes of fully augmented links with diagrams lying on higher genus surfaces in 3-manifolds besides the 3-sphere, provided the augmented link admits a reflective symmetry.

Our starting point is the recent paper ~\cite{FuchsPurcellStewart}, which generalises the construction of fully augmented links to lie on higher genus surfaces in ambient 3-manifolds distinct from the 3-sphere, by starting with a circle packing on the conformal boundary of any infinite volume hyperbolic 3-manifold, and building a fully augmented link from the given data. Such links were used to produce knots and links converging geometrically to infinite volume hyperbolic 3-manifolds. 
This paper complements~\cite{FuchsPurcellStewart} by examining more carefully the geometry of generalised fully augmented links. However, here we start with a topological definition, and show that the defining conditions lead to geometric consequences, including the existence of a circle packing in an appropriate setting. We prove that, as in the classical setting, the links have bounded cusp geometry, and we give lower bounds on their volumes. 

We define fully augmented links carefully in Section~\ref{Sec:ConstFullAug}. Briefly, they consist of link components embedded in the projection plane, ``augmented'' by unknots encircling two strands (called \emph{knot strands}), possibly with half-twists; see Figure~\ref{Fig:HalfTwist}.
The fully augmented links that we construct lie on a doubled 3-manifold: Given a compact 3-manifold $M$ with a surface $\Sigma\subseteq\bdy M$, let $\calD_\Sigma(M)$ denote the double; see Definition~\ref{Def:DoubledManifold}. The fully augmented links will lie in a neighbourhood of $\Sigma\subset \calD_\Sigma(M)$. The key point is that by embedding the links symmetrically in a 3-manifold with reflective symmetry, we obtain rigid geometric properties that can be read off the diagram, just as in the classical case.


\begin{theorem}[Cusp shapes of fully augmented links]\label{Thm:CuspShapesMain}
Let $M$ be a compact orientable 3-manifold with $\Sigma\subset \bdy M$ a surface. Let $L$ be a fully augmented link on $\Sigma$ in $\calD_\Sigma(M)$. Then the torus boundary components of $\calD_\Sigma(M)-N(L)$ that correspond to link components of $\bdy N(L)$ are tiled by rectangles with opposite white and black sides satisfying the following conditions.
\begin{enumerate}
\item[(1)] Crossing circles meet two rectangles, with a longitude running along two black sides and meridian running along one white side if there is no half-twist, and running across the diagonal of a rectangle if there is a half-twist.
\item[(2)] Knot strands meeting $n_j$ crossing circles (counted with multiplicity) meet $2n_j$ rectangles, with a meridian formed from running along two black sides and a longitude running along white sides where there is no half-twist, and a diagonal of the rectangle for each half-twist.
\item[(3)] When the interior of $\calD_\Sigma(M)-N(L)$ is hyperbolic, there exists a horoball expansion about the cusps of $L$ such that each black side has length one, and each white side has length at least one.
\end{enumerate}
\end{theorem}

Combined with Dehn filling theorems, \refthm{CuspShapesMain} has implications for geometries of 3-manifolds obtained by filling these links. For example, the 6-Theorem will give results on when they are hyperbolic~\cite{Agol:Bounds, Lackenby:Word}; see also~\cite{FuterPurcell:Exceptional}. 

\begin{theorem}[Lower volume bounds]\label{Thm:VolumeMain}
Let $M$ be a compact orientable 3-manifold with $\Sigma\subset \bdy M$. Let $L$ be a hyperbolic fully augmented link on $\Sigma$ in $\calD_{\Sigma}(M)$ with $c$ crossing circles. Then the volume of $\calD_{\Sigma}(M)-L$ satisfies:
\[\vol(\calD_{\Sigma}(M)-L) \geq 2\,\voct\,(c-\chi(M))\]
Here $v_{oct}=3.66386\dots$ is the volume of a hyperbolic regular ideal octahedron.
\end{theorem}

Combining \refthm{VolumeMain} with \refthm{CuspShapesMain} and results on volume change under Dehn filling, we obtain lower bounds on volumes of the 3-manifolds obtained by Dehn filling fully augmented links; see~\cite{FKP:DehnFilling, Purcell:Volumes}. A result along these lines is \refthm{VolumeDehnFilling} below. 

We are also able to generalise the upper bounds on volumes of fully augmented links in thickened surfaces. We restrict to those that are \emph{cellular}, meaning the complementary regions of the diagram are all discs. 

\begin{theorem}[Upper volume bounds, virtual setting]\label{Thm:VolumeVirtualUpper}
Let $\Sigma$ be a surface of genus $g$, and let $M$ be the manifold $\Sigma\times[-1,0]$. Obtain $\calD_{\Sigma\times\{0\}}(M)=\Sigma\times[-1,1]$, doubled across the component $\Sigma\times\{0\}$ of $M$. Let $L$ be a cellular fully augmented link in $\calD_{\Sigma\times\{0\}}(M)$ with $c$ crossing circles. Then the volume of $\calD_{\Sigma\times\{0\}}(M)-L$ is bounded above as follows.
\[\mbox{When $g=1$, } \vol(\Sigma\times(-1,1)- L)\leq 10\vtet\,c,\]
\[\mbox{when $g>1$, } \vol((\Sigma\times [-1,1])- L)\leq 6\,\voct\,c. \]
Here, $\vtet$ is the volume of a hyperbolic regular ideal tetrahedron, and $\voct$ is the volume of a hyperbolic regular ideal octahedron. In the case $g=1$, $\Sigma\times\{\pm 1\}$ are cusps, and when $g>1$, we take the hyperbolic structure in which $\Sigma\times\{\pm 1\}$ are totally geodesic.

Moreover, this theorem is sharp in the case $g=1$. 
\end{theorem}

Theorem~\ref{Thm:VolumeVirtualUpper} should be compared to~\cite[Theorem~1.4]{KalfagianniPurcell}. Our proof is similar to that proof, but the links under consideration are different. The existence of such upper bounds is significant, as the second author has recently shown that for some 3-manifolds $M$, there are no comparable upper volume bounds~\cite{Reid:NoUpperVolumeBound}.

Note that there are other recent generalisations of fully augmented links lying in thickened surfaces, due to Adams~\emph{et al}~\cite{adamsetal:2021Generalized}, Kwon~\cite{kwon2020fully}, and Kwon and Tham~\cite{kwon2020hyperbolicity}. The links of those papers are subsets of the more broad class that we study here.

\subsection{Motivation}

The proofs of the main theorems are very similar to those in the classical setting, although the ambient 3-manifolds are very different. The main purpose of this article is to point out these extensions, so that fully augmented links can be applied in new settings.

There are applications to links in well-known doubled manifolds:
\begin{itemize}
\item \emph{Virtual fully augmented links}, lying on $\Sigma\times\{0\}$ in $\calD_{\Sigma\times\{0\}}(\Sigma\times [0,1])$; see \refcor{LowerBoundVirtual}. 
\item \emph{Fully augmented links in doubled handlebodies}, lying on $\bdy H$ for a handlebody $H$ in $\calD_{\bdy H}(H)$; see \refcor{LowerBoundHandlebody}. These have connections to shadow links. 
\end{itemize}

While the proofs are very similar to those of~\cite{FuterPurcell:Exceptional, FKP:DehnFilling} (see also~\cite{Purcell:AugLinksSurvey} and~\cite[Chapter~7]{Purcell:HyperbolicKnotTheory}), we strive to include enough detail here to ensure the paper is self-contained.

\subsection{Acknowledgements}
This work was supported in part by the Australian Research Council, grant DP210103136.


\section{Constructing fully augmented links}\label{Sec:ConstFullAug}

In this section, we define fully augmented links. We start with a very general definition as in~\cite{FuchsPurcellStewart}, on any surface in any 3-manifold. 

\begin{definition}\label{Def:FALNoTwists}
Fix a 3-manifold $M$, and fix $\Sigma$ an embedded surface in $M$. A \emph{fully augmented link on $\Sigma$ without half-twists} is a link $L$ embedded in a regular neighbourhood of $\Sigma$ consisting of components $K_1, \dots, K_m$ and $C_1, \dots, C_n$ that satisfies the following properties:
\begin{enumerate}
\item[(1)] Each component $K_i$ is embedded in $\Sigma$, for $1\leq i\leq m$
\item[(2)] Each $C_j$ bounds a disc $D_j$ in $M$ such that $D_j$ intersects $\Sigma$ transversely in a single arc, and $D_j$ meets the union $\coprod_{i} K_i$ in exactly two points, for $1 \leq j \leq n$.
\item[(3)] A projection of $L$ to $\Sigma$ yields a 4-valent \emph{diagram graph} on $\Sigma$. We require this diagram to be connected. 
\end{enumerate}
The components $C_j$ are called \emph{crossing circles}, and the discs $D_j$ are called \emph{crossing discs}. The components $K_k$ lie on the \emph{projection surface}.
\end{definition}

We also allow a \emph{half-twist} at a crossing circle, obtained by cutting along $D_j$ and regluing so that the two points of intersection of $\coprod_i K_i$ with $D_j$ are swapped; see \reffig{HalfTwist}. Note after performing a half-twist, condition~(1) in \refdef{FALNoTwists} is typically not satisfied.
\begin{figure}
  \centering
  \includegraphics{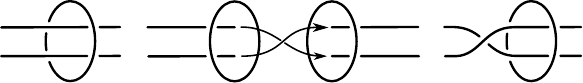}
  \caption{Cutting and regluing at a crossing circle to add a half-twist.}
  \label{Fig:HalfTwist}
\end{figure}

\begin{definition}\label{Def:FAL}
A \emph{fully augmented link on a surface $\Sigma$} is obtained from a fully augmented link on $\Sigma$ without half-twists by inserting a single half-twist at some or no crossing circles. 
\end{definition}

\begin{definition}\label{Def:DoubledManifold}
Let $M$ be a compact orientable 3-manifold with distinguished boundary component $\Sigma \neq \emptyset \subset \bdy M$. The \emph{doubled manifold (along $\Sigma$)}, denoted $\calD(M)$ (or $\calD_\Sigma(M)$ when we wish to emphasize $\Sigma$), is the manifold obtained by taking two copies of $M$, one with reversed orientation, and identifying them along $\Sigma$. 
\end{definition}

\begin{proposition}\label{Prop:Reflection}
Suppose $M$ is a compact 3-manifold with distinguished boundary component(s) $\Sigma \neq \emptyset \subset \bdy M$, and $L$ is a fully augmented link on $\Sigma$ in the double $\calD_\Sigma(M)$. Then the link complement $\calD_\Sigma(M)-L$ admits a reflective symmetry. The reflection fixes a surface pointwise that agrees with $\Sigma$ everywhere aside from a neighbourhood of each crossing disc meeting a half-twist. It reflects each crossing disc $D_j$ through the arc of intersection of $D_j$ with $\bdy M$ in $\calD(M)$.
\end{proposition}

\begin{proof}
If the link has no half-twists, the symmetry is obtained by the reflection through $\Sigma$ in the double $\calD(M)$. Note that it fixes $\Sigma$ pointwise, hence fixes pointwise each component $K_i$ on the projection plane. Arranging crossing discs $D_j$ to be perpendicular to $\bdy M$ gives the required reflection in $D_j$.

At a half-twist, the reflection through $\Sigma$ in $\calD(M)$ still takes a crossing disc $D_j$ to itself, but it reverses the direction of the crossing. Apply a twist homeomorphism: Cut along $D_j$ and rotate by $2\pi$ in the opposite direction. This is a homeomorphism of the complement of $C_j$, equal to the identity outside a neighbourhood of $D_j$, and it takes the link back to the original with the half-twist in the original direction. Observe that the composition of the reflection with the twist homemomorphism is an orientation reversing symmetry; this is the required reflection. 
\end{proof}

\begin{definition}\label{Def:LReflectiveSymmetry}
A fully augmented link as in \refprop{Reflection} is called a \emph{fully augmented link with reflective symmetry}. The surface fixed pointwise by the reflection is called the \emph{surface of reflection}. Note it is obtained from the surface of projection by cutting along each arc meeting a crossing disc at a half-twist and regluing. 
\end{definition}

\begin{example}\label{Example:Virtual}
Let $\Sigma$ be a closed orientable surface. A fully augmented link on the surface $\Sigma \times \{0\}$ in $\Sigma\times [-1,1]$ is called a \emph{virtual fully augmented link}. Note it has reflective symmetry, by setting $M=\Sigma\times[0,1]$ in \refdef{LReflectiveSymmetry}. These are examples of the fully augmented links of~\cite{adamsetal:2021Generalized, kwon2020fully, kwon2020hyperbolicity}. 
\end{example}

\begin{example}\label{Example:Handlebody}
Let $H$ be a handlebody, with $\Sigma=\bdy H$. Then the double $\calD(H)$ is a connected sum of copies of $S^2\times S^1$. A fully augmented link with reflective symmetry lies on $\bdy H \subset \calD(H)$. 
\end{example}

\begin{example}\label{Example:KnotComplement}
The exterior of an open neighbourhood of a classical knot in the 3-sphere is a well-known example of a nontrivial compact irreducible 3-manifold. We obtain a fully augmented link with reflective symmetry by adding the diagram of a fully augmented link to the torus boundary component of the knot exterior, and taking the double. Observe that such a manifold will typically admit two essential tori, obtained by doubling the boundary of a regular neighbourhood of the original knot.
\end{example}

In this paper, we are most interested in hyperbolic components of a JSJ decomposition of a fully augmented link complement. In the following sections, we derive consequences to geometry in the hyperbolic setting.

\section{Geometry and surfaces}\label{Sec:GeomSurfaces}


\begin{proposition}\label{Prop:GeodesicReflectiveSurface}
Let $L$ be a fully augmented link with reflective symmetry in $\calD(M)$. Suppose that a component of $L$ lies in a hyperbolic component of the JSJ decomposition of $\calD(M)-L$. Then the reflective surface in the link complement is a totally geodesic hyperbolic surface.

Each crossing disc $D_j-L$ within such a component is also totally geodesic. 
\end{proposition}

\begin{proof}
Since $\calD(M)-L$ admits a reflective symmetry, by the Equivariant Torus Theorem, first proved by Holzmann~\cite{Holzmann}, any essential torus in the JSJ decomposition can be isotoped either to be taken to a distinct disjoint essential torus, or taken to itself by the reflection. Thus when we cut along the essential tori of the JSJ decomposition, any component meeting the reflection surface continues to admit a reflective symmetry, fixing the reflection surface pointwise. If such a component is hyperbolic, then it follows from Mostow--Prasad rigidity that the surface fixed pointwise is totally geodesic in the hyperbolic metric; see, for example~\cite{Leininger, MenascoReid}. 
Finally, each crossing disc $D_j-L$ in $\calD(M)-L$ is a 3-punctured sphere, which has a unique totally geodesic hyperbolic structure~\cite{adams1985thrice}. 
\end{proof}

\begin{lemma}\label{Lem:RightAngles}
Suppose a link component of $L$ lies in a hyperbolic component of the JSJ decomposition of $\calD(M)-L$. Then the reflection surface and the crossing discs $D_i-L$ meet at right angles. 
\end{lemma}

\begin{proof}
  The reflective symmetry preserves the reflection surface pointwise, but reflects each crossing disc in the arc where the crossing disc meets the reflection surface. Since the crossing disc is taken to itself under the reflection, it must meet the reflection surface at right angles. 
\end{proof}

\begin{lemma}\label{Lem:RectangleCusps}
Let $L_j$ be a component of $L$. Then the torus $\bdy N(L_j)$ is tiled by rectangles. Each rectangle has two opposite sides coming from the intersection of $\bdy N(L_j)$ with the reflection surface; colour these white. The other two sides come from intersections with crossing discs; colour these black. 
\end{lemma}

\begin{proof}
By \reflem{RightAngles}, the reflection surface and the crossing discs meet at right angles in a hyperbolic component of the JSJ decomposition. Even more generally, any link component $L_j$ will meet both crossing discs and the reflection surface, with the intersections of these surfaces meeting $\bdy N(L_j)$ in 4-valent vertices. Thus these two surfaces cut $\bdy N(L_j)$ into rectangles. 
\end{proof}

\begin{lemma}\label{Lem:CirclePattern}
Suppose $L_j$ lies in a hyperbolic component of the JSJ decomposition. Consider the lifts of rectangles of \reflem{RectangleCusps} to $\HH^3$. 
Each rectangle has two opposite white sides and two black sides. The black sides lie over a white circle that is tangent to the two white sides. If we scale such that the length of a black side is one, the length of a white side will be at least one. 
\end{lemma}

\begin{proof}
This follows from the fact that the black crossing discs are cut by the reflection surface into ideal triangles. A black triangle in the rectangle of \reflem{RectangleCusps} meets two vertical white planes in two of its sides. The third side meets a lift of a component of the reflection surface that runs between these two ideal edges on the vertical white planes. Because the reflection surface is totally geodesic, its lift to $\HH^3$ is a geodesic plane in $\HH^3$, which has boundary a white circle tangent to the white vertical planes making up sides of the rectangle. See \reffig{Circles}.

\begin{figure}
  \includegraphics{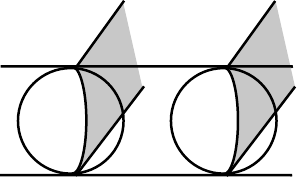}
  \caption{Rectangle with circle white sides}
  \label{Fig:Circles}
\end{figure}

Finally, observe that if we scale so that a black side has length one, then each circle tangent to the white sides has diameter one. Because the reflection surface is embedded, these two circles will be disjoint. Thus the length of the white side is at least one. 
\end{proof}

As a converse to \reflem{CirclePattern}, in~\cite{FuchsPurcellStewart} a hyperbolic fully augmented link is constructed by starting with a circle packing on $\bdy M$ that is hyperbolisable and gluing circles appropriately. 

\section{Cusp shapes and cusp areas}
This section presents results on cusps of fully augmented links in $\calD(M)$. The results generalise similar results in the classical case; see for example \cite{FuterPurcell:Exceptional}, sections~2.2 and~2.3.

\begin{proposition}\label{Prop:BoundaryRectangles}
Suppose $M$ is a compact orientable 3-manifold with boundary $\bdy M$. Suppose $L$ is a fully augmented link with reflective symmetry, lying on $\Sigma\subset \bdy M$ in the double $\calD_{\Sigma}(M)$, and let $L_j$ be a component of $L$.
\begin{enumerate}
\item If $L_j$ comes from a crossing circle, $\bdy N(L_j)$ meets exactly two rectangles of \reflem{RectangleCusps}.
\item If $L_j$ is a strand of the knot embedded in the reflection surface, $\bdy N(L_j)$ meets $2\times n_j$ such rectangles, where $n_j$ is the number of times $L_j$ runs through a crossing circle (counted with multiplicity). 
\end{enumerate}
\end{proposition}

\begin{proof}
A crossing circle meets exactly one crossing disc $D_j$, which cuts $\bdy N(L_j)$ into an annulus with two black boundary components, coming from where the cusp meets the black side.
If it does not enclose a half-twist, the crossing circle also meets the reflection surface exactly twice in two meridians, which cuts the annulus into two rectangles. If it meets a half-twist, it is still the case that one encounters two rectangles when traversing the boundary of a crossing disc: the second is the reflected copy of the first in the double. 

As for a component in the reflection surface, note that the reflection surface meets $\bdy N(L_j)$ twice, cutting it into two annuli, which are swapped by the reflective symmetry. Each time the link component runs through a crossing circle, it meets a crossing disc. For each annulus, it therefore meets $n_j$ crossing discs, dividing each annulus into $n_j$ rectangles. Because there are two annuli, the total number of rectangles is $2n_j$. 
\end{proof}

When there are no half-twists in the diagram, the rectangles of \refprop{BoundaryRectangles} form fundamental domains for the torus $\bdy N(L_j)$ without shearing. When there are half-twists, the fundamental domains are sheared:

\begin{lemma}\label{Lem:HalfTwistShear}
Let $D_j$ denote the crossing disc bounded by a crossing circle $C_j$. Then adding a half-twist at $C_j$ is realised by a homeomorphism that cuts along the 3-punctured disc $D_j$, and reglues after a half-twist. This glues the top half of the disc on one side to the bottom half on the other, and vice versa. Thus it introduces a shearing. 
\end{lemma}

\begin{proof}
When the component is hyperbolic, the gluing is by isometry of the 3-punctured disc; see Adams~\cite{adams1985thrice} or Purcell~\cite{Purcell:AugLinksSurvey}, particularly Figure~9 in that paper. Even in the nonhyperbolic case, the combinatorics of the gluing is as stated, and the effect is the same.
\end{proof}

Choose any orientation on the components of the link that lie in the reflection surface. Let $C_j$ be a crossing circle, bounding crossing disc $D_j$. Consider one of the two strands of $L$ that run through $C_j$. By \refprop{BoundaryRectangles} (or perhaps more accurately its proof), locally the disc $D_j$ subdivides a horospherical torus into four rectangles: A rectangle $B_1$ lying on one side of the reflection plane (say above), glued by reflection to a rectangle $B_1'$ below the reflection plane, and a rectangle $B_2$ above the reflection plane on the opposite side of $D_j$ glued to a rectangle $B_2'$ below. Depending on how the strands of the link connect in the reflection surface, a half-twist will have one of the following effects. 

\begin{lemma}\label{Lem:HalfTwistShearKnotStrand}
Adding a half-twist to the diagram at a crossing circle $C_j$ has the following effect on the cusp tiling.
\smallskip

\noindent(1) Suppose both link components passing through $C_j$ belong to the same component $K$, and run in the same direction. Let $B_1$, $B_1'$ and $B_2$, $B_2'$ denote subrectangles meeting the crossing disc $D_j$ the first time the link component runs through, and $B_3$, $B_3'$, $B_4$, $B_4'$ denote subrectangles the second time. 
Then after inserting a half-twist, $K$ is split into two disjoint link components. The cusp for one component is tiled by $B_1$, $B_1'$ and $B_4$, $B_4'$, the other by $B_2$, $B_2'$, $B_3$, $B_3'$, with the gluing shifted up or down along the black faces corresponding to the crossing disc $D_j$, depending on the direction of the half-twist. See \reffig{BoundaryRectangle}.

\smallskip

\noindent(2) Suppose both link components passing through $C_j$ belong to the same component $K$ running through in opposite directions. Then there is still one link component after adding a half-twist at $C_j$, but the gluing shifts by one step along the black faces corresponding to $D_j$ each time $K$ passes through $C_j$. The rectangular block between the two points where $L$ passes through $C_j$ will be reversed. See \reffig{BoundaryRectangle}.
\smallskip

\noindent(3) Suppose the two link components running through $C_j$ belong to two different link components $J,K\subset L$. Then after adding a half-twist, the resulting cusp torus is the concatenation of the two blocks with gluing shifted along the black faces where they join. See \reffig{BoundaryRectangle}. 
\end{lemma}

\begin{figure}
  \centering
  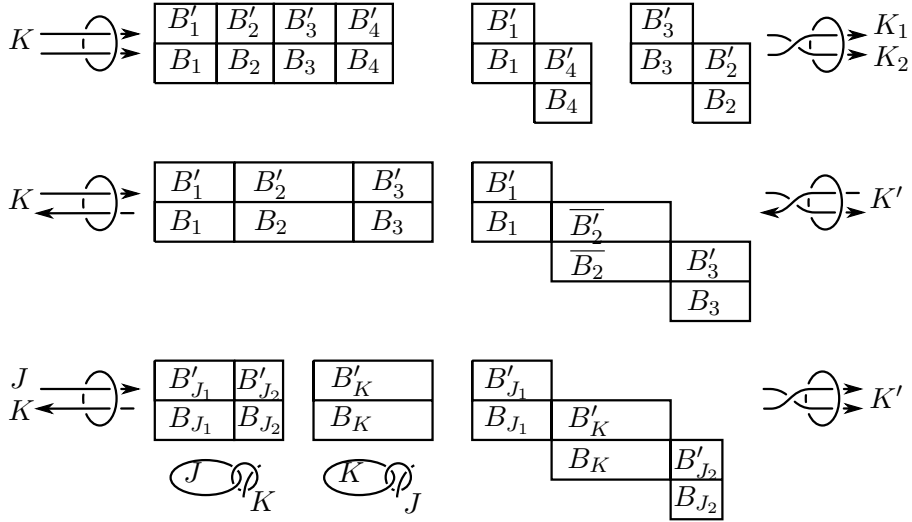
  \caption{How the tiling of horospherical cusp tori changes when adding a half-twist at a crossing circle.}
  \label{Fig:BoundaryRectangle}
\end{figure}

\begin{proof}
When both strands passing through $C_j$ run in the same direction and belong to the same component $K$, adding a half-twist at $C_i$ will correspond to splitting $K$ into two components $K_1$ and $K_2$. The block of boundary rectangles $B_1$ will now be glued to $B_4'$, the block corresponding $B_4$ will be glued to $B_1'$. This corresponds to shifting the gluing up or down depending on the direction of the half-twist. The tail of $B_4$ is glued to the head of $B_1$ with no change. Similarly the tail of $B_3$ will be glued to the head of $B_2'$.
		
If the strands pass through $C_j$ in opposite directions for the same link component, adding a half-twist will not increase the number of components, but it will change the direction of the strand on one side of $C_j$. Thus adding a half-twist at $C_j$ reverses the block of boundary rectangles $B_2$; denote this by $\overline{B_2}$, and similarly $\overline{B_2'}$. The new cusp tiling is given by $B_1$ glued to $B_1'$ across the white reflection side, and to $\overline{B_2'}$ across the black side. Similarly $\overline{B_2}$ is glued to $B_3'$. 
		
In the case that there are two different components $J$ and $K$ passing through $C_j$, we can always orient in the same direction when passing though $C_j$. Let $B_1$ and $B_2$ be the two blocks of boundary rectangles which tile the cusp of $J$ and meet $C_j$ between $B_1$ and $B_2$, and let $B_3$ be the block of boundary rectangles which tile $K$ meeting $C_j$ at its head and tail. Then adding a half-twist yields a single component tiled by gluing the tail of $B_1$ to the head of $B_3'$ and the tail of $B_3$ to $B_2'$, shifted up or down depending on the direction of the half-twist.
\end{proof}

Finally we include results about horoball sizes in the hyperbolic case. The previous lemmas, along with \reflem{CirclePattern}, give information about the Euclidean structure of a cusp corresponding to $\bdy N(L_j)$ for $L_j$ in a hyperbolic component of the JSJ decomposition of $\calD(M)-N(L)$. However, this is only determined up to scale. By fixing horoballs we are able to fix the area, or scaling, of the cusp as well.

The following proposition is an almost immediate generalisation of a similar result in the classical case, with proofs almost identical to those in \cite[Section~3.2]{FuterPurcell:Exceptional}; see also \cite[Chapter~7]{Purcell:HyperbolicKnotTheory}. There, it was assumed throughout that the complement consisted of hyperbolic polyhedra, and the proofs refer to the polyhedra. Here, we do not have ideal polyhedra. However, we step through the proof to show that the key point is not the polyhedra, but the reflective symmetry.

\begin{definition}\label{Def:Midpoint}
Let $e$ be an edge of a hyperbolic ideal triangle with ideal vertices $v_1$, $v_2$, $v_3$, where the endpoints of $e$ lie on $v_1$ and $v_2$. The \emph{midpoint} of $e$ is defined to be the point where a geodesic from the third ideal vertex $v_3$ meets $e$ at a right angle. 
\end{definition}

\begin{proposition}\label{Prop:CanonicalEdge}
Let $L$ be a hyperbolic fully augmented link in a doubled manifold $\calD(M)$. Then there exists a horoball expansion about the cusps of $\calD(M)-L$ such that the midpoint of every edge is a point of tangency of the horospherical tori.
\end{proposition}

\begin{proof}
Lift a fixed cusp to the point at infinity in the universal cover $\HH^3$ such that a rectangle of \reflem{RectangleCusps} has black side length exactly one. There are two black ideal triangles in the rectangle. For each, the midpoint is at Euclidean height one.

Now expand horoball neighbourhoods of the cusps of $L$ uniformly, so that each has the same distance from the midpoint along triangle edges. If the horoballs expand to distance zero from the midpoint, or Euclidean height one, we are finished. So suppose this is not possible. Then the horoballs at the corners of rectangles have diameter strictly less than one, but there must exist a tangency of these horoballs elsewhere. Conjugate so that one of the tangent horoballs is at infinity. Scaling so that black sides of rectangles have length one, the other horoball $H$ will have diameter strictly greater than one.

Consider the centre of $H$. First, note that it cannot lie in the interior of the rectangle, for because its diameter is strictly greater than one and the black side of the rectangle has length one, in that case $H$ would meet a vertical plane corresponding to a white side of the rectangle in a closed curve. But then the reflection in the white side would take $H$ to a horoball intersecting $H$, contradicting the fact that we expanded horoballs keeping them embedded. 

So $H$ must have centre $p$ on a white side $W$ of the rectangle, in the interior of the white side. The point $p$ must project to a point on $L$ (since we have only expanded horoballs about $L$ so far). Thus $H$ meets white and black faces in rectangles. In particular, the white side of the rectangle will be tangent to some other white circle $W'$ at $p$, with these two white planes forming opposite sides of a rectangle at $p$. Because the reflection surface is embedded, and lifts to white circles, this white circle $W'$ is disjoint from the other white circles and lines on $\bdy_\infty\HH^3$; in particular it has diameter at most one. Additionally, there are two black triangles meeting the point $p$ that form the black sides of the rectangle at $p$. Recall that we have expanded the horosphere $H$ about $p$ so that it does not meet the midpoint of each of these black triangles.

\begin{figure}
  \includegraphics{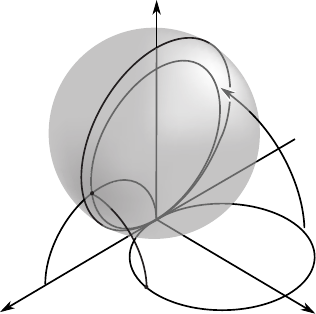}
  \caption{If $H$ is centred at a point $p$ on a white face $W$, and $H$ has diameter greater than one, then it must contain the midpoint of an edge through $p$. This is Figure~7.17 of~\cite{Purcell:HyperbolicKnotTheory}}
  \label{Fig:Horosphere}
\end{figure}

One of the ideal edges of a black rectangle at $p$ lies on the intersection of $W$ with a geodesic plane corresponding to a black triangle. This is a semicircle on $W$. The midpoint of this edge is obtained by dropping a perpendicular line from the semicircle to a point on the circle $W'$. But the fact that the diameter of $H$ is greater than one and the diameter of $W'$ is at most one means that any perpendicular from $W'$ to $W$ lies in the interior of $H$, and thus $H$ contains the midpoint in its interior. See \reffig{Horosphere}. This is a contradiction. So there is no such horosphere $H$, and horospheres can be expanded to midpoints of black triangles while their interiors remain embedded. In this expansion, black sides have length exactly one. 
\end{proof}

\begin{proof}[Proof of \refthm{CuspShapesMain}]
The fact that cusps are tiled as in~(1) and~(2) of the statement of the theorem follows from \refprop{BoundaryRectangles}, with shearing as in  \reflem{HalfTwistShear} and \reflem{HalfTwistShearKnotStrand}. The side lengths follow from \refprop{CanonicalEdge}.
\end{proof}

\section{Lower bounds on volume}

We now turn to the proof of Theorem~\ref{Thm:VolumeMain}. It begins similarly to the classical setting. The key feature we need is the existence of totally geodesic surfaces, using the following theorem from Miyamoto~\cite{Miyamoto}. 

\begin{theorem}[Miyamoto] \label{Thm:Miyamoto}
  A hyperbolic 3-manifold $M$ with totally geodesic boundary has volume
  \[\vol(M) \geq -v_{oct}\,\chi (M)\]
  where $v_{oct}=3.66386...$ is the volume of a hyperbolic regular ideal octahedron. The bound is an equality if and only if $M$ decomposes into regular ideal octahedra. \qed
\end{theorem}

The reflection surface $R$ for $\calD(M)-L$ is totally geodesic, by \refprop{GeodesicReflectiveSurface}. If we cut along this surface, Miyamoto's \refthm{Miyamoto} applies to show that the volume is at most $-\voct\chi\left( (\calD(M)-L)\cut R\right)$.

It remains to bound $\chi\left((\calD(M)-L)\cut R\right)$. In the classical setting, the Euler characteristic is obtained by drilling tubes from a ball; see~\cite[Proposigion~3.1]{FKP:DehnFilling}. In the setting of doubled manifolds, we consider more carefully the construction.

\begin{lemma}\label{Lem:VolEulerChar}
  The Euler characteristic $\chi\left((\calD(M)-L)\cut R\right)$ satisfies:
  \[ \chi((\calD(M)-L)\cut R) = 2(\chi(M)-c)
  \] 
\end{lemma}

\begin{proof}
Cut $\calD(M)-L$ along $R$. This cuts each crossing disc into two pieces. Further cut along these half crossing discs.

If there are no half-twists, the surface $R$ is the surface $\bdy M$, and cutting along $R$ splits $\calD(M)-L$ into two copies of $M$ with $L\cap M$ drilled out.

If there are half-twists, these differ from the link without half-twist only by adjusting the gluing of the half-discs. Hence cutting along half crossing discs again splits the manifold into two copies of $M$ with $L\cap M$ drilled, further cut along half crossing discs.

In both cases, after cutting along $R$ and crossing discs, the result is homeomorphic to two copies of $M$. The link $L$ only marks the boundary $\bdy M$, and does not affect Euler characteristic of the cutting. Thus the cut manifold has Euler characteristic $2\chi(M)$.

Regluing a pair of half crossing discs has the topological effect of adding a 1-handle, in both cases with or without half-twists. Since each crossing circle meets two half-discs, after regluing, the total effect on Euler characteristic is to subtract $2c$.
Thus $\chi((\calD(M)-L)\cut R) = 2(\chi(M)-c)$.
\end{proof}

\begin{proof}[Proof of Theorem~\ref{Thm:VolumeMain}]
Miyamoto's \refthm{Miyamoto} implies that the volume is at most $-\voct\chi\left((\calD(M)-L)\cut R\right)$, which is $-2\voct(\chi(M)-c)$ by \reflem{VolEulerChar}. 
\end{proof}

\begin{corollary}\label{Cor:LowerBoundHandlebody}
Suppose $M=H_g$ is a handlebody of genus $g$, and $L$ is a hyperbolic fully augmented link in $\calD(H_g)$. Then the volume satisfies
\[\vol(\calD(H_g)-L) \geq 2\,\voct\,(g+c-1). \]
\end{corollary}

\begin{proof}
The Euler characteristic of $M=H_g$ in this case is $g-1$. 
\end{proof}

\begin{corollary}\label{Cor:LowerBoundVirtual}
Suppose $\Sigma$ is a surface of genus $g$, and $M=\Sigma\times[-1,0]$, with $\calD(M)$ doubled along the boundary component $\Sigma\times\{0\}$, so that a fully augmented link $L$ in $\calD(M)$ is a \emph{virtual} fully augmented link. If such a link is hyperbolic, its volume satisfies
\[ \vol(\Sigma\times[-1,1]-L) \geq 2\,\voct\,(2g+c-2). \]
\end{corollary}

\begin{proof}
The Euler characteristic of $M=\Sigma\times[-1,0]$ is $2g-2$. 
\end{proof}

Note Adams~\emph{et~al} also obtain lower volume bounds on classes of virtual fully augmented links in~\cite{adamsetal:2021Generalized}. The bounds of \refcor{LowerBoundVirtual} agree when $\Sigma$ is a torus, and are stronger when $\Sigma$ has higher genus. 

\section{Upper volume bounds for virtual links}

In \cite{Purcell:AugLinksSurvey}, there are both upper and lower bounds on volumes of classical fully augmented links. In the case of fully augmented links in a doubled manifold, the existence of upper bounds depends on the ambient 3-manifold. Indeed, the second author has recently shown that for some 3-manifolds $M$, there are no upper volume bounds~\cite{Reid:NoUpperVolumeBound}. However, in the setting of virtual links, there are upper bounds. This was essentially proven in~\cite{KalfagianniPurcell}. We repeat the argument here for completeness. Note that Adams~\emph{et al} also considered volumes of subcases in \cite{adamsetal:2021Generalized}, but again \refthm{VolumeVirtualUpper} is more general. 

\begin{proof}[Proof of \refthm{VolumeVirtualUpper}]
The idea of the proof is to cut the link complement into pieces, and bound the hyperbolic volume of each piece. Start by cutting along the reflection surface, which slices crossing discs into half-discs. Next slice along each half disc. The result is two copies of $\Sigma\times(-1,0]$, with remnants of $L$ carved out of the boundary $\Sigma\times\{0\}$. Each black half-crossing-disc has become two ideal triangles on the surface $\Sigma\times\{0\}$. Adjacent to these across ideal edges are faces of the white reflection surface.

In the case $g=1$, form $2c$ ideal tetrahedra by coning each black triangle in $\Sigma\times(-1,0]$ to $\Sigma\times\{-1\}$; these have the form $T\times(-1,0]$, for $T$ a black ideal triangle. Note that $T\times\{-1\}$ will be an ideal vertex in the hyperbolic structure, hence the result is an ideal tetrahedron. Similarly cone the ideal triangles on the reflected copy of $M$, to obtain a total of $4c$ ideal tetrahedra so far, corresponding to black faces.

What remains are regions $W\times(-1,1)$, which come from white regions $W$ of the diagram that are reflected. Because the diagram is cellular, each region $W$ is an ideal polygon. Perform stellar subdivision on $W\times(-1,1)$. That is, let $V\in W$ be a point in $W$. Then $V\times(-1,1)$ is an ideal edge. Subdivide into ideal tetrahedra by adding faces of the form $E\times(-1,1)$, where $E$ is a line in $W$ from an ideal vertex of $W$ to the point $V$. Each resulting tetrahedron meets exactly one of the original edges of $W$. Each edge of $W$ came from the intersection of the reflection surface and the crossing discs; after cutting crossing discs, there are $6c$ such intersections. Thus we obtain $6c$ ideal tetrahedra from stellar subdividing. In total, there are $4c+6c=10c$ ideal tetrahedra making up the complement $\Sigma\times(-1,1)-L$, so the volume is at most $6c$ times the maximum volume $\vtet$ of an ideal tetrahedron. 

For genus $g>1$, the decomposition is similar, but ``coning'' gives hyperideal polyhedra; see \reffig{Hyperideal}. Black triangles are coned to have one hyperideal vertex and three ideal vertices. Each white face $W$ is stellar subdivided into tetrahedra with two hyperideal and two ideal vertices. 

\begin{figure}
  \centering
  \includegraphics{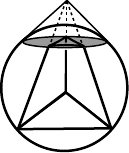}
  \caption{A hyperideal tetrahedron in $\HH^3$. The hyperideal vertex may be viewed as lying beyond $\bdy_\infty \HH^3$. }
  \label{Fig:Hyperideal}
\end{figure}

By work of Adams, Calderon, and Mayer~\cite[Corollary~3.4]{AdamsCalderonMayer}, the maximum volume of a tetrahedron with two ideal vertices and two hyperideal vertices is $\voct/2$. Thus the hyperideal tetrahedra coming from the white regions $W$ contribute at most $6c \cdot \voct/2$ to the volume.

As for the black triangles from crossing discs, to use the bound of Adams, Calderon, and Mayer, we identify two of these across a black face and again stellar subdivide into three hyperideal tetrahedra. This creates $4c/2\cdot3=6c$ additional tetrahedra with two hyperideal vertices and two ideal vertices. Thus the total volume bound is $12c\cdot \voct/2 = 6c\voct$. 

To prove sharpness in the case $g=1$, we present an example. In fact, this follows from work of Agol and Thurston in \cite[Appendix]{Lackenby:Vol}. The link shown in \reffig{ChainLinkTorus} has an infinite cover given by the ``infinite chain link fence'' of that paper. Agol and Thurston showed that the stellar subdivision described above gives six regular ideal tetrahedra per white face, and four regular ideal tetrahedra per black triangle. The covering map glues all white faces to a single face, and all black triangles to the four triangles coming from the single crossing circle shown.
\end{proof}

\begin{figure}[h]
  \centering
  \includegraphics{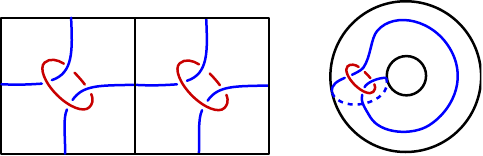}
  \caption{Left: Tiles making up a finite part of the infinite chain link fence. Right: A link on a torus whose cover is the infinite chain link fence.}
    \label{Fig:ChainLinkTorus}
\end{figure}

\section{Dehn fillings}

If we perform $1/n$ Dehn filling on a crossing circle of a fully augmented link, we obtain a link with the crossing circle removed and $2n$ crossings inserted. This can be used, for example, to build virtual knots and links, and other broad families of knots and links in doubled 3-manifolds. The work above allows us to give lower bounds on the volumes of such links.

\begin{proposition}\label{Prop:SlopeLength}
Let $L$ be a hyperbolic fully augmented link in a doubled 3-manifold $\calD(M)$. There is a horosphere expansion about cusps of $L$ so that the slope yielding $n$ crossings has length at least $n$. 
\end{proposition}

\begin{proof}
Each crossing circle cusp is tiled by two rectangles, and we can choose the horosphere expansion as in \refprop{CanonicalEdge} such that each black side has length one, and white sides have length at least one. If $n=2k$ is even, the slope yielding $n$ crossings is $1/k$ on a crossing circle with no half-twist. If $n=2k+1$ is odd, the crossing circle has a half-twist, and the slope is $1/k$. 

If there is no half-twist, the meridian follows one white side and the longitude two black sides. Thus the $1/k$ slope, which follows one meridian and two longitudes, has length $\sqrt{1+4k^2}>2k=n$.

If there is a half-twist, the longitude still follows two black sides, but the meridian now is sheared to follow one white side and one black side. Thus the $1/k$ slope runs along $2k+1$ black sides and $1$ white side in the rectangular tiling of the cusp, to have length $\sqrt{1+(2k+1)^2} > n$. 
\end{proof}

\begin{theorem}\label{Thm:VolumeDehnFilling}
Let $L$ be a fully augmented link in a doubled 3-manifold $\calD(M)$ with $c$ crossing circles. Suppose the knot or link $K$ is obtained from $L$ by adding at least $m\geq 7$ crossings in each twist region. Then the volume of $K$ is bounded below as follows:
\[ \vol(\calD(M)-K) \geq \left(1- \left( \frac{2\pi}{m} \right)^2 \right)^{3/2} \cdot 2\,\voct\,(c-\chi(M)) \]
\end{theorem}

\begin{proof}
By Futer--Kalfagianni--Purcell~\cite[Theorem~1.1]{FKP:DehnFilling}, the volume of the filled manifold is bounded by
\[ \left(1-\left(\frac{2\pi}{\ell_{\mathrm{min}}}\right)^2\right)^{3/2}\vol(\calD(M)-L) \]
where $\ell_{\mathrm{min}}$ is the minimum length over all slopes of Dehn fillings to produce $K$. By \refprop{SlopeLength}, $\ell_{\mathrm{min}}\geq m$. By \refthm{VolumeMain}, the volume $\vol(\calD(M)-L)$ is at least $2\,\voct\,(c-\chi(M))$.
\end{proof}

\bibliographystyle{amsplain}
\bibliography{biblio}

\end{document}